\documentclass[11pt, twoside]{article}

\usepackage{ascmac}
\usepackage{amsmath}
\usepackage{amssymb}
\usepackage{amsfonts}
\usepackage{amsthm}

\newtheorem{thm}{Theorem}
\newtheorem{lem}{Lemma}

\makeatletter
\@addtoreset{equation}{section}

\makeatother

\title{Analytic Continuation of Multiple Zeta-Functions and\\the Asymptotic Behavior at Non-Positive Integers}
\author{TOMOKAZU ONOZUKA}

\begin{document}
\date{}
\maketitle

\section{Introduction}

The Euler-Zagier multiple zeta function $\zeta_d(s_1, \cdots,s_d)$ is defined by
\begin{equation}
  \zeta_d (s_1,\cdots,s_d) := \sum_{m_1=1}^{\infty}\cdots\sum_{m_d=1}^{\infty}\frac{1}{m_1^{s_1}(m_1+m_2)^{s_2}\cdots(m_1+ \cdots+m_d)^{s_d}}
\label{eq:def}
\end{equation}
where $s_i\ (i=1, \cdots,d)$ are complex variables. Matsumoto \cite{mat1} proved that the series (\ref{eq:def}) is absolutely convergent in
\[
  \left\{(s_1,\cdots,s_d)\in \mathbb{C}^d\ |\ \Re (s_d(d-k+1))>k\ \ (k=1, \cdots,d)\right\}
\]
where $s_d(n)=s_n+s_{n+1}+ \cdots+s_d\ \ (n=1, \cdots,d)$. Akiyama, Egami and Tanigawa \cite{aki1} and Zhao \cite{zhao} proved the meromorphic continuation to the whole space independently. Akiyama, Egami and Tanigawa used the Euler-Maclaurin summation formula and Zhao used generalized functions to prove the analytic continuation. Later, Matsumoto \cite{mat2} also proved the analytic continuation using Mellin-Barnes integral formula. 

The function $\zeta_d(s_1, \cdots,s_d)$ has singularities on
\begin{eqnarray}
\begin{cases}
  s_d=1,\\
  s_{d-1}+s_d=2,1,0,-2,-4, \cdots,\\
  s_d(d-j+1)\in \mathbb{Z}_{\leq j}\ (j=3,4, \cdots ,d),
\end{cases}
\label{eq:sin}
\end{eqnarray}
where $\mathbb{Z}_{\leq j}$ is the set of integers less than or equal to $j$; $\mathbb{Z}_{\geq j}$ is defined similarly. Therefore $(-r_1, \cdots.-r_d)\in \mathbb{Z}^d_{\leq 0}$ lies on the set of singularities. Moreover, it is an indeterminacy of  $\zeta_d (s_1,\cdots,s_d)$. For example, Sasaki \cite{sasa} proved that
\begin{align}
  \lim_{s_3\to 0}\lim_{s_2\to 0}\lim_{s_1\to 0}\zeta_3(s_1,s_2,s_3)=&-\frac{3}{8} \label{eq:exam1},\\
  \lim_{s_1\to 0}\lim_{s_2\to 0}\lim_{s_3\to 0}\zeta_3(s_1,s_2,s_3)=&-\frac{1}{4}. \label{eq:exam2}
\end{align}
Since $(0,0,0)$ is an indeterminacy of $\zeta_3 (s_1,s_2,s_3)$, (\ref{eq:exam1}) and (\ref{eq:exam2}) give different values.

Akiyama, Egami and Tanigawa \cite{aki1} defined the regular values by
\[
  \nonumber\zeta_d(-r_1, \cdots,-r_d):=\lim_{s_1\to -r_1} \cdots \lim_{s_d\to -r_d}\zeta_d(s_1, \cdots,s_d),
\]
and Akiyama and Tanigawa \cite{aki2} considered the reverse and central values given by
\begin{align}
  \nonumber\zeta_d^R(-r_1, \cdots,-r_d)&:=\lim_{s_d\to -r_d} \cdots \lim_{s_1\to -r_1}\zeta_d(s_1, \cdots,s_d),\\
  \nonumber\zeta_d^C(-r_1, \cdots,-r_d)&:=\lim_{\varepsilon\to 0}\zeta_d(-r_1+\varepsilon, \cdots,-r_d+\varepsilon),
\end{align}
respectively. Further, Sasaki \cite{sasa} generalized the regular and reverse values. He defined multiple zeta values for coordinatewise limits by
\[
  \zeta_d(\overset {i_1}{-r_1}, \cdots,\overset{i_d}{-r_d}):=\lim_{\underset{i_j=d}{s_j\to-r_j}} \cdots\lim_{\underset{i_j=1}{s_j\to-r_j}}\zeta_d(s_1, \cdots,s_d),
\]
where $\{i_1, \cdots,i_d\}=\{1, \cdots,d\}$. He obtained all multiple zeta values of depth 3 for coordinatewise limits. In addition, he treated the multiple zeta values of depth 4 for coodinatewise limits in \cite{sasa2}. On the other hand, Kamano \cite{kama} considered the regular, reverse and central values of the multiple Hurwitz zeta funcions. Komori \cite{komo} considered more general multiple zeta functions, and he obtained multiple zeta values at non-positive integers given by
\begin{align}
  \nonumber \zeta_d( \overset{w}{-\mbox{\boldmath $r$}})&=
  \lim_{z_{w^{-1}(d)}\to -r_{w^{-1}(d)}}\cdots\lim_{z_{w^{-1}(1)}\to -r_{w^{-1}(1)}}\zeta_d(z_1, \cdots,z_d),\\
  \nonumber \zeta_d( \underset{\mbox{\boldmath $\theta$}}{-\mbox{\boldmath $r$}})&=
  \zeta_d(\underset{\theta_1}{-r_1},\cdots, \underset{\theta_d}{-r_d})=
  \lim_{\delta\to0}\zeta_d(-r_1+\delta\theta_1,\cdots,-r_d+\delta\theta_d),
\end{align}
where $ -\mbox{\boldmath $r$}=(-r_1, \cdots,-r_d)\in \mathbb{Z}_{\leq0}^d$, $w\in\mathfrak{S}_d$ and $ \mbox{\boldmath $\theta$}=(\theta_1, \cdots,\theta_d)\in \mathbb{C}^d$. To obtain these values by Komori's method, we need to compute generalized multiple Bernoulli numbers.

In the present paper, we prove two theorems. Theorem 1 gives the meromorphic continuation of the multiple zeta function to the whole space. The meromorphic continuation was already proved. Proof of Theorem 1 is similar to the proof of meromorphic continuation in \cite{zhao}. In \cite{zhao}, Zhao used the theory of generalized functions \cite{gel} to prove the meromorphic continuation. On the other hand, to prove Theorem 1, we do not use the theory of generalized functions but integration by parts. In Theorem 2, we prove asymptotic behavior near the non-positive integers. Until now, we have been able to get only 2 kinds of the limit values, $\zeta( \overset{w}{-\mbox{\boldmath $r$}})$ and $\zeta_d( \underset{\mbox{\boldmath $\theta$}}{-\mbox{\boldmath $r$}})$. Using Theorem 2, we can compute not only  $\zeta( \overset{w}{-\mbox{\boldmath $r$}})$, $\zeta_d( \underset{\mbox{\boldmath $\theta$}}{-\mbox{\boldmath $r$}})$ but also various different types of limit values. In fact, by Theorem 2, we can compute, for example,
\begin{equation}
  \lim_{\varepsilon\to 0}\zeta_3(\varepsilon^2,\varepsilon, \varepsilon)=-\frac{1}{3}.\label{eq:exam3}
\end{equation}
This limit value is not contained in the above 2 kinds of values, however by Theorem 2, we can compute this value.

The author would like to express his thanks to Professor Kohji Matsumoto for valuable advice and comments.

\section{Main theorems}

In this section, we state two theorems.

Let $B_m$ be the $m$th Bernoulli number, and $B(x,y)$ be the beta function. For $ (m_1, \cdots , m_d) $$ \in $$ \mathbb{Z}_{\geq0}^d$, $(p_1, \cdots,p_d)$$\in \mathbb{Z}_{\geq0}^d$ and $(\varepsilon_1, \cdots, \varepsilon_d)\in \mathbb{C}^d$, let $m_d(n)$, $p_d(n)$ and $\varepsilon_d(n)$ be $m_n+m_{n+1}+ \cdots+m_d$, $p_n+p_{n+1}+ \cdots+p_d$ and $\varepsilon_n+\varepsilon_{n+1}+ \cdots+\varepsilon_d$ respectively. In addition, Pochhammer symbol $(a)_n$ is defined by $(a)_n:=\Gamma(a+n)/\Gamma(a)$. In this paper, symmetric group $\mathfrak{S}$ is defined by $\bigl\{\sigma|\sigma:\{2, \cdots,d\}\to\{2, \cdots,d\},\ \sigma\ {\rm is\ a\ bijective\ function}\bigr\}$.

\begin{thm}
For $d\geq 2$ and $n_1, \cdots,n_d\in\mathbb{Z}_{\geq 0}$, $\zeta_d(s_1, \cdots,s_d)$ can be continued meromorphically to 
\[
  \left\{(s_1, \cdots,s_d)\in \mathbb{C}^d\ |\ \Re(s_d(j))>d-j-n_j\ (j=1, \cdots,d),\ \Re(s_{j-1})>-n_j-1\ (j=2, \cdots,d)\right\},
\]
and $\zeta_d(s_1, \cdots,s_d)$ can be represented by
\begin{align}
  \nonumber\zeta_d(&s_1, \cdots,s_d)\\
\nonumber=&\frac{1}{\Gamma(s_1)\cdots\Gamma(s_d)}\sum_{k=0}^{n_1}\sum_{p_1+ \cdots+p_d=k}\frac{B_{p_1}\cdots B_{p_d}}{p_1! \cdots p_d!}\frac{1}{s_d(1)-d+k}\prod_{j=2}^{d}B(s_d(j)-d+j+p_d(j)-1,s_{j-1})\\
  \nonumber&+\frac{1}{\Gamma(s_1)\cdots\Gamma(s_d)}\int_0^1 x_1^{s_d(1)-d+n_1}F_{\varphi_a}(x_1)dx_1\\
  &+\frac{1}{\Gamma(s_1)\cdots\Gamma(s_d)}\int_1^{\infty} \frac{x_1^{s_d(1)-d}}{e^{x_1}-1}F_{\psi_a}(x_1)dx_1,
  \label{eq:thm1}
\end{align}
where 
\begin{align}
  \nonumber F_{f_a}&(x_1)\\
  \nonumber:=&\sum_{(a_2, \cdots,a_d)}\sum_{m=1}^d \sum_{\substack{\sigma(2)< \cdots<\sigma(m)\\
  \sigma(m+1)< \cdots<\sigma(d)}}\sum_{k_{\sigma(2)}=0}^{n_{\sigma(2)}} \cdots\sum_{k_{\sigma(m)}=0}^{n_{\sigma(m)}}\left\{\prod_{j=2}^{m}(-1)^{k_{\sigma(j)}}(u_{\sigma(j)}+1)_{k_{\sigma(j)}+1}^{-1}\left(\frac{1}{2}\right)^{u_{\sigma(j)}+k_{\sigma(j)}+1} \right\}\\
  \nonumber&\times\left\{\prod_{j=m+1}^{d}(-1)^{n_{\sigma(j)}+1}(u_{\sigma(j)}+1)_{n_{\sigma(j)}+1}^{-1}\right\}
  \nonumber\int_0^{\frac{1}{2}} \cdots\int_0^{\frac{1}{2}}\left(\prod_{j=m+1}^{d}x_{\sigma(j)}^{u_{\sigma(j)}+n_{\sigma(j)}+1}\right)\\
  &\quad\qquad\left[\frac{d^{k_{\sigma(2)}}}{dx_{\sigma(2)}^{k_{\sigma(2)}}} \cdots\frac{d^{k_{\sigma(m)}}}{dx_{\sigma(m)}^{k_{\sigma(m)}}}\frac{d^{n_{\sigma(m+1)}+1}}{dx_{\sigma(m+1)}^{n_{\sigma(m+1)}+1}}\cdots\frac{d^{n_{\sigma(d)}+1}}{dx_{\sigma(d)}^{n_{\sigma(d)}+1}}
f_a(x_1, \cdots,x_d)\right]_{\substack{x_{\sigma(2)}=\frac{1}{2}\\ \vdots\\x_{\sigma(m)}=\frac{1}{2}}}\nonumber
  \nonumber  dx_{\sigma(m+1)} \cdots dx_{\sigma(d)},
\end{align}

\begin{align}
\nonumber\varphi_a(x_1, \cdots,x_d)&:=\left(\prod_{j=2}^{d}(1-x_j)^{v_j}\right)
  \sum_{k=n_1+1}^{\infty}\sum_{p_1+ \cdots+p_d=k}\frac{B_{p_1} \cdots B_{p_d}}{p_1! \cdots p_d!}x_1^{k-n_1-1}t_2^{p_d(2)} \cdots t_d^{p_d(d)},\\
\nonumber\psi_a(x_1, \cdots,x_d)&:=\prod_{j=2}^{d}(1-x_J)^{v_j}\frac{x_1t_2 \cdots t_j}{e^{x_1t_2 \cdots t_j}-1}.
\end{align}
Here, the summation $\sum_{(a_2, \cdots,a_d)}$ runs all combinations of $a_j=0$ or $1\ (j=2, \cdots,d)$, and $\sum_{\substack{\sigma(2)< \cdots<\sigma(m)\\ \sigma(m+1)< \cdots<\sigma(d)}}$ runs all $\sigma\in\mathfrak{S}$ satisfying $\sigma(2)< \cdots<\sigma(m)$ and $\sigma(m+1)< \cdots<\sigma(d)$, and $u_j$, $v_j$ and $t_j$ are defined by
\begin{align}
\nonumber u_j&:=\left\{ \begin{array}{ll}
s_d(j)-d+j-2&(a_j=0),\\
s_{j-1}-1&(a_j=1),\\
\end{array} \right.
\nonumber &v_j:=\left\{ \begin{array}{ll}
s_{j-1}-1&(a_j=0),\\
s_d(j)-d+j-2&(a_j=1),\\
\end{array} \right.\\
\nonumber t_j&:=\left\{ \begin{array}{ll}
x_j&(a_j=0),\\
1-x_j&(a_j=1).\\
\end{array} \right.
\end{align}
The function $\zeta_d(s_1, \cdots,s_d)$ has possible singularities on
\[
  \left\{(s_1, \cdots,s_d)\in \mathbb{C}^d\ |\ s_d(j)\in \mathbb{Z}_{\leq d-j+1},\ s_j\in \mathbb{Z}_{\leq 0}\ (j=1, \cdots,d)\right\}.
\]
\end{thm}

Using Theorem 1, we can obtain the following Theorem 2.

\begin{thm}
Suppose that $\varepsilon_j \neq 0,\ \varepsilon_d(j) \neq0\ (j=1, \cdots,d)$, $| \varepsilon_1|+ \cdots+|\varepsilon_d|\leq\frac{1}{2}$ and $ |\varepsilon_k/ \varepsilon_d(j)|\ll1$ as $(\varepsilon_1, \cdots, \varepsilon_d)\to(0, \cdots,0)\  (j=1, \cdots,d,\ k=j, \cdots,d)$. Then for $m_j\in \mathbb{Z}_{\geq0}\ (j=1, \cdots,d)$, we have
\begin{align}
  \nonumber&\zeta_d(-m_1+ \varepsilon_1, \cdots,-m_d+ \varepsilon_d)=(-1)^{m_d}m_d! \\
  \nonumber&\times\sum_{p_1+ \cdots+p_d=d+M}\frac{B_{p_1} \cdots B_{p_d}}{p_1! \cdots p_d!}\prod_{j=2}^{d}h(-m_d(j)-d+j+p_d(j)-1,-m_d(j-1)-d+j+p_d(j)-1)\\
  \nonumber&\times \frac{[ \varepsilon_d(j)]_{-m_d(j)-d+j+p_d(j)-1}}{[ \varepsilon_d(j-1)]_{-m_d(j-1)-d+j+p_d(j)-1}}+\sum_{j=1}^{d}O( \varepsilon_j)\\
  \nonumber=&(-1)^{m_d}m_d! \sum_{\substack{p_1+ \cdots+p_d=d+M\\\vspace{-1mm} \\-m_d(j)-d+j+p_d(j)<2\  \mbox{or}\\-m_d(j-1)-d+j+p_d(j)\geq2\ (2\leq\forall j\leq d)}}
  \frac{B_{p_1} \cdots B_{p_d}}{p_1! \cdots p_d!}\prod_{j=2}^{d} \frac{[ \varepsilon_d(j)]_{-m_d(j)-d+j+p_d(j)-1}}{[ \varepsilon_d(j-1)]_{-m_d(j-1)-d+j+p_d(j)-1}}+\sum_{j=1}^{d}O( \varepsilon_j)
\end{align}
as $( \varepsilon_1, \cdots, \varepsilon_d)\to(0, \cdots,0)$, where
\begin{align}
  \nonumber M&:=m_1+ \cdots+m_d,\\
  \nonumber[a]_n&:=\left\{ \begin{array}{ll}
    a(n-1)!&(n\geq1),\\
    (-1)^n(-n)!^{-1}&(n<0),\\
  \end{array} \right.\\
  \nonumber h(m,n)&:=\left\{ \begin{array}{ll}
    0&(m\geq1>n),\\
    1&(otherwise).\\
  \end{array} \right.
\end{align}
\end{thm}

In Theorem 2, $\varepsilon_j \ (j=1, \cdots,d)$ should satisfy  $ |\varepsilon_k/ \varepsilon_d(j)|\ll1\ (j=1, \cdots,d,\ k=j, \cdots,d)$. Let us think about this condition. If $ |\varepsilon_k/ \varepsilon_d(j)|\to\infty$, then $\varepsilon_d(j)$ tends to $0$ rapidly. By (\ref{eq:sin}), $s_j+ \cdots+s_d=-M$ is a singular locus. Therefore, when $|\varepsilon_k/ \varepsilon_d(j)|\to\infty$, the point $(-m_1+ \varepsilon_1, \cdots,-m_d+ \varepsilon_d)$ approximates asymptotically to a singular locus. Hence,  $ |\varepsilon_k/ \varepsilon_d(j)|\ll1$ means geometrically that $(-m_1+ \varepsilon_1, \cdots,-m_d+ \varepsilon_d)$ does not approximate asymptotically to a singular locus.

\section{Examples}
By Theorem 2, we can compute various multiple zeta values at non-positive integers. Let us see some examples.

In the case $d=2$, we have
\begin{align}
  \nonumber \zeta_2( \varepsilon_1, \varepsilon_2)&=\frac{1}{3}+\frac{1}{24}\cdot\frac{ \varepsilon_2}{ \varepsilon_1+ \varepsilon_2}+\sum_{j=1}^{2}O( \varepsilon_j),\\
  \nonumber \zeta_2(-1+ \varepsilon_1, \varepsilon_2)&=\frac{1}{24}+\sum_{j=1}^{2}O( \varepsilon_j),\\
  \nonumber \zeta_2( \varepsilon_1, -1+\varepsilon_2)&=\frac{1}{12}+\sum_{j=1}^{2}O( \varepsilon_j),\\
  \nonumber \zeta_2(-1+ \varepsilon_1,-1+ \varepsilon_2)&=\frac{1}{360}+\frac{1}{720}\cdot\frac{ \varepsilon_2}{ \varepsilon_1+ \varepsilon_2}+\sum_{j=1}^{2}O( \varepsilon_j).
\end{align}

In the case $d=3$, we have
\begin{align}
  \nonumber \zeta_3( \varepsilon_1, \varepsilon_2, \varepsilon_3)&=-\frac{1}{4}-\frac{1}{24}\cdot\frac{ \varepsilon_3}{ \varepsilon_2+ \varepsilon_3}-\frac{1}{24}\cdot\frac{ \varepsilon_2+2\varepsilon_3}{ \varepsilon_1+\varepsilon_2+ \varepsilon_3}+\sum_{j=1}^{3}O( \varepsilon_j),\\
  \nonumber \zeta_3(-1+ \varepsilon_1, \varepsilon_2, \varepsilon_3)&=-\frac{17}{720}-\frac{1}{144}\cdot\frac{ \varepsilon_3}{ \varepsilon_2+ \varepsilon_3}+\frac{1}{720}\cdot\frac{ -\varepsilon_2+3\varepsilon_3}{ \varepsilon_1+\varepsilon_2+ \varepsilon_3}+\sum_{j=1}^{3}O( \varepsilon_j),\\
  \nonumber \zeta_3( \varepsilon_1, -1+\varepsilon_2, \varepsilon_3)&=-\frac{19}{360}+\frac{1}{360}\cdot\frac{ \varepsilon_2}{\varepsilon_1+ \varepsilon_2+ \varepsilon_3}+\sum_{j=1}^{3}O( \varepsilon_j),\\
  \nonumber \zeta_3( \varepsilon_1, \varepsilon_2, -1+\varepsilon_3)&=-\frac{3}{40}-\frac{1}{720}\cdot\frac{4 \varepsilon_2+3\varepsilon_3}{\varepsilon_1+ \varepsilon_2+ \varepsilon_3}+\sum_{j=1}^{3}O( \varepsilon_j).
\end{align}
Note that the example (\ref{eq:exam3}) comes from the first example of the above, taking $ \varepsilon_1= \varepsilon^2$ and $ \varepsilon_2= \varepsilon_3= \varepsilon$.

In the case $d=4$, we have
\begin{align}
  \nonumber \zeta_4( \varepsilon_1, \varepsilon_2, \varepsilon_3, \varepsilon_4)=
  \frac{1}{5}+\frac{1}{36}\cdot\frac{ \varepsilon_4}{ \varepsilon_3+ \varepsilon_4}&+\frac{1}{48}\cdot\frac{ \varepsilon_3+2\varepsilon_4}{\varepsilon_2+ \varepsilon_3+ \varepsilon_4}+\frac{1}{720}\cdot\frac{ 19\varepsilon_2+ 33\varepsilon_3+52\varepsilon_4}{\varepsilon_1+\varepsilon_2+ \varepsilon_3+ \varepsilon_4}\\
  \nonumber&+\frac{1}{144}\cdot\frac{ \varepsilon_4(\varepsilon_2+\varepsilon_3+ \varepsilon_4)}{ (\varepsilon_3+\varepsilon_4)(\varepsilon_1+\varepsilon_2+ \varepsilon_3+ \varepsilon_4)}+\sum_{j=1}^{4}O( \varepsilon_j).
\end{align}

\section{Lemmas}

To prove Theorem 1 and Theorem 2, in this section, we prove several lemmas. In this section, suppose that $\Re(s_k)>1\ (k=1, \cdots,d)$.

\begin{lem}

Let $I_1$ be an interval on $ \mathbb{R}$ and $f(x_1, \cdots,x_d)$ be of class $C^{\infty}$ on $I_1\times [0,\frac{1}{2}] \times \cdots \times [0,\frac{1}{2}] \subset \mathbb{R}^d$. Then for all $(a_2, \cdots,a_d)$, $x_1\in I_1$ and every $k=2, \cdots,d$, we have
\begin{align}
  \nonumber\int_0^{\frac{1}{2}} &\cdots\int_0^{\frac{1}{2}}\left(\prod_{j=2}^{d}x_j^{u_j}\right)f(x_1, \cdots,x_d)dx_2 \cdots dx_d\\
  \nonumber=&\sum_{m=1}^{k} \sum_{\substack{\sigma(k;2)< \cdots<\sigma(k;m)\\
  \sigma(k;m+1)< \cdots<\sigma(k;k)}}\left\{\prod_{j=2}^{m}(u_{\sigma(k;j)}+1)^{-1}
  \left(\frac{1}{2}\right)^{u_{\sigma(k;j)}+1} \right\}
  \left\{\prod_{j=m+1}^{k}(-1)(u_{\sigma(k;j)}+1)^{-1}\right\}\\
  \label{eq:lem1}&\times\int_0^{\frac{1}{2}} \cdots\int_0^{\frac{1}{2}}
  \left(\prod_{j=m+1}^{k}x_{\sigma(k;j)}^{u_{\sigma(k;j)}+1}\right)
  \left(\prod_{j=k+1}^{d}x_j^{u_j}\right)\\
  \nonumber&\hspace{2cm}\left[\frac{d}{dx_{\sigma(k;m+1)}}\cdots\frac{d}{dx_{\sigma(k;k)}}
  f(x_1, \cdots,x_d)\right]_{\substack{x_{\sigma(k;2)}=\frac{1}{2}\\ \vdots\\x_{\sigma(k;m)}=\frac{1}{2}}}
  dx_{\sigma(k;m+1)} \cdots dx_{\sigma(k;k)} dx_{k+1} \cdots dx_d,
\end{align}
where $\sigma(k;)$ is an element of the group given by \[
  \mathfrak{S}_k=\bigl\{\sigma(k;)\ \bigl|\ \sigma(k;):\{2, \cdots,k\}\to\{2, \cdots,k\},\ \sigma(k;)\ {\rm is\ a\ bijective\ function}\bigr\}.
\]
\end{lem}

\begin{proof}
In the case of $k=2$, using integration by parts with respect to $x_2$ on the left-hand side of (\ref{eq:lem1}), we have
\begin{align}
  \nonumber\int_0^{\frac{1}{2}} \cdots&\int_0^{\frac{1}{2}}\left(\prod_{j=2}^{d}x_j^{u_j}\right)f(x_1, \cdots,x_d)dx_2 \cdots dx_d\\
  \label{eq:prf1}=&(u_2+1)^{-1}\left(\frac{1}{2}\right)^{u_2+1}\int_0^{\frac{1}{2}} \cdots\int_0^{\frac{1}{2}}\left(\prod_{j=3}^{d}x_j^{u_j}\right)\left[f(x_1, \cdots,x_d)\right]_{x_2=\frac{1}{2}}dx_3 \cdots dx_d\\
\nonumber &-(u_2+1)^{-1}\int_0^{\frac{1}{2}} \cdots\int_0^{\frac{1}{2}}x_2^{u_2+1}\left(\prod_{j=3}^{d}x_j^{u_j}\right)\left[\frac{\partial}{ \partial x_2}f(x_1, \cdots,x_d)\right]dx_2 \cdots dx_d.
\end{align}
The first term on the right-hand side of (\ref{eq:prf1}) is the term corresponding to $m=2$ of (\ref{eq:lem1}), and the second term is the term corresponding to $m=1$ of (\ref{eq:lem1}).

Suppose that Lemma 1 holds for $k-1$. Using integration by parts with respect to $x_k$ on the right-hand side of (\ref{eq:lem1}), we have
\begin{align}
  \nonumber &(u_k+1)^{-1}\left(\frac{1}{2}\right)^{u_k+1}
  \left[\frac{d}{dx_{\sigma(k-1;m+1)}}\cdots\frac{d}{dx_{\sigma(k-1;k-1)}}
  f(x_1, \cdots,x_d)\right]_{\substack{x_{\sigma(k-1;2)}=\frac{1}{2}\\ \vdots\\x_{\sigma(k-1;m)}=\frac{1}{2}\\x_{k}=\frac{1}{2}}}\\
  \label{eq:prf2} -&(u_k+1)^{-1}\int_0^{\frac{1}{2}}x_{k}^{u_k+1}
  \left[\frac{ \partial}{ \partial x_k}\frac{d}{dx_{\sigma(k-1;m+1)}}\cdots\frac{d}{dx_{\sigma(k-1;k-1)}}
  f(x_1, \cdots,x_d)\right]_{\substack{x_{\sigma(k-1;2)}=\frac{1}{2}\\ \vdots\\x_{\sigma(k-1;m)}=\frac{1}{2}}}dx_k.
\end{align}
The first term of (\ref{eq:prf2}) is the term corresponding to $k=\sigma(k;m)$ of (\ref{eq:lem1}), and the second term of (\ref{eq:prf2}) is the term corresponding to $k=\sigma(k;k)$ of (\ref{eq:lem1}).
\end{proof}

\begin{lem}
Let $f(x_1, \cdots,x_d)$ be as in Lemma 1. Then for all $(a_2, \cdots,a_d)$, $x_1\in I_1$, we have
\begin{align}
  \nonumber&\int_0^{\frac{1}{2}} \cdots\int_0^{\frac{1}{2}}
  \left(\prod_{j=2}^{d}x_j^{u_j}\right)f(x_1, \cdots,x_d)dx_2 \cdots dx_d\\
  \nonumber=&\sum_{m=1}^d \sum_{\substack{\sigma(2)< \cdots<\sigma(m)\\ \sigma(m+1)< \cdots<\sigma(d)}}\sum_{k_{\sigma(2)}=0}^{n_{\sigma(2)}} \cdots\sum_{k_{\sigma(m)}=0}^{n_{\sigma(m)}}
  \left\{\prod_{j=2}^{m}(-1)^{k_{\sigma(j)}}(u_{\sigma(j)}+1)_{k_{\sigma(j)}+1}^{-1}\left(\frac{1}{2}\right)^{u_{\sigma(j)}+k_{\sigma(j)}+1} \right\}\\
  \nonumber&\times\left\{\prod_{j=m+1}^{d}(-1)^{n_{\sigma(j)}+1}(u_{\sigma(j)}+1)_{n_{\sigma(j)}+1}^{-1}\right\}\\
  \nonumber&\times\int_0^{\frac{1}{2}} \cdots\int_0^{\frac{1}{2}}\left(\prod_{j=m+1}^{d}x_{\sigma(j)}^{u_{\sigma(j)}+n_{\sigma(j)}+1}\right)\\
  &\hspace{1.0cm}\left[\frac{d^{k_{\sigma(2)}}}{dx_{\sigma(2)}^{k_{\sigma(2)}}} \cdots\frac{d^{k_{\sigma(m)}}}{dx_{\sigma(m)}^{k_{\sigma(m)}}}\frac{d^{n_{\sigma(m+1)}+1}}{dx_{\sigma(m+1)}^{n_{\sigma(m+1)}+1}}\cdots\frac{d^{n_{\sigma(d)}+1}}{dx_{\sigma(d)}^{n_{\sigma(d)}+1}}\right.
f(x_1, \cdots,x_d)\Biggr]_{\substack{x_{\sigma(2)}=\frac{1}{2}\\ \vdots\\x_{\sigma(m)}=\frac{1}{2}}}dx_{\sigma(m+1)} \cdots dx_{\sigma(d)}\label{eq:lem2},
\end{align}
where $n_2, \cdots,n_d\in \mathbb{Z}_{\geq0}$, $\sigma$ are as in Theorem 1.
\end{lem}

\begin{proof}
Induction on $n_2+ \cdots+n_d$.$\\$
In the case $n_2= \cdots=n_d=0$, (\ref{eq:lem2}) is valid by Lemma 1.

Suppose that (\ref{eq:lem2}) holds for $(n_2, \cdots,n_d)$, and let us prove (\ref{eq:lem2}) for $(n_2, \cdots,n_k+1, \cdots,n_d)$. The right-hand side of (\ref{eq:lem2}) is divided into two terms,
\begin{equation}
 \sum_{m=1}^d \sum_{\substack{\sigma(2)< \cdots<\sigma(m)\\ \sigma(m+1)< \cdots<\sigma(d)\\ \vspace{-0.5mm} \\ k\in\{ \sigma(2), \cdots, \sigma(m)\}}} \sum_{k_{\sigma(2)}=0}^{n_{\sigma(2)}} \cdots\sum_{k_{\sigma(m)}=0}^{n_{\sigma(m)}}+
  \sum_{m=1}^d \sum_{\substack{\sigma(2)< \cdots<\sigma(m)\\ \sigma(m+1)< \cdots<\sigma(d)\\ \vspace{-0.5mm} \\ k\in\{\sigma(m+1), \cdots,\sigma(d)\}}}\sum_{k_{\sigma(2)}=0}^{n_{\sigma(2)}} \cdots\sum_{k_{\sigma(m)}=0}^{n_{\sigma(m)}}.
  \label{eq:1siki}
\end{equation}
The first term of (\ref{eq:1siki}) has no integral of $x_k$, and the second term of (\ref{eq:1siki}) has an integral of $x_k$. Using integration by parts with respect to $x_k$ on the second term of (\ref{eq:1siki}), we have
\begin{align}
  \nonumber(u_k+n_k+2)^{-1}\left\{\left(\frac{1}{2}\right)^{u_k+n_k+2}
  \left[\frac{d^{k_{\sigma(2)}}}{dx_{\sigma(2)}^{k_{\sigma(2)}}} \cdots\frac{d^{k_{\sigma(m)}}}{dx_{\sigma(m)}^{k_{\sigma(m)}}}\frac{d^{n_{\sigma(m+1)}+1}}{dx_{\sigma(m+1)}^{n_{\sigma(m+1)}+1}}\cdots\frac{d^{n_{\sigma(d)}+1}}{dx_{\sigma(d)}^{n_{\sigma(d)}+1}}
  f(x_1, \cdots,x_d)\right]_{\substack{x_{\sigma(2)}=\frac{1}{2}\\ \vdots\\x_{\sigma(m)}=\frac{1}{2}\\ x_k=\frac{1}{2}}}\right.\\
  -\left.\int_0^{\frac{1}{2}}x_k^{u_k+n_k+2}
  \left[\frac{d^{k_{\sigma(2)}}}{dx_{\sigma(2)}^{k_{\sigma(2)}}} \cdots\frac{d^{k_{\sigma(m)}}}{dx_{\sigma(m)}^{k_{\sigma(m)}}}\frac{d^{n_{\sigma(m+1)}+1}}{dx_{\sigma(m+1)}^{n_{\sigma(m+1)}+1}}\cdots\frac{d^{n_k+2}}{dx_k^{n_k+2}}\cdots\frac{d^{n_{\sigma(d)}+1}}{dx_{\sigma(d)}^{n_{\sigma(d)}+1}}
  f(x_1, \cdots,x_d)\right]_{\substack{x_{\sigma(2)}=\frac{1}{2}\\ \vdots\\x_{\sigma(m)}=\frac{1}{2}}}dx_k\right\}.
  \label{eq:2siki}
\end{align}
Using (\ref{eq:1siki}) and (\ref{eq:2siki}), we find (\ref{eq:lem2}) for $(n_2, \cdots,n_k+1, \cdots,n_d)$.
\end{proof}

\begin{lem}
$\varphi_a(x_1, \cdots,x_d)$ is $C^{\infty}$ on $[0,1]\times[0,\frac{1}{2}]\times \cdots\times[0,\frac{1}{2}]\subset \mathbb{R}^d$.
\end{lem}

\begin{proof}
Since $(1-x_j)^{v_j}\ (j=2, \cdots,d)$ are $C^{\infty}$ on $[0,1]\times[0,\frac{1}{2}]\times \cdots\times[0,\frac{1}{2}]$, what we have to prove is that
\[
\sum_{k=n_1+1}^{\infty}\sum_{p_1+ \cdots+p_d=k}\frac{B_{p_1} \cdots B_{p_d}}{p_1! \cdots p_d!}x_1^{k-n_1-1}t_2^{p_d(2)} \cdots t_d^{p_d(d)}
\]
is $C^{\infty}$ on $[0,1]\times[0,\frac{1}{2}]\times \cdots\times[0,\frac{1}{2}]$. Clearly, we have
\begin{multline}
  \sum_{k=n_1+1}^{\infty}\sum_{p_1+ \cdots+p_d=k}\frac{B_{p_1} \cdots B_{p_d}}{p_1! \cdots p_d!}x_1^{k-n_1-1}t_2^{p_d(2)} \cdots t_d^{p_d(d)}\\
  =\frac{\prod_{j=1}^{d}\frac{x_1t_2 \cdots t_j}{e^{x_1t_2 \cdots t_j}-1}-
  \sum_{k=0}^{n_1}\sum_{p_1+ \cdots+p_d=k}\frac{B_{p_1} \cdots B_{p_d}}{p_1! \cdots p_d!}x_1^{k}t_2^{p_d(2)} \cdots t_d^{p_d(d)}}{x_1^{n_1+1}}\label{eq:lem3}.
\end{multline}
We prove that the right-hand side of (\ref{eq:lem3}) is $C^{\infty}$. The numerator of the right-hand side is $C^{\infty}$, so the right-hand side is $C^{\infty}$ except for $x_1=0$. We can find that $x_1=0$ is a removable singularity by observing the left-hand side of (\ref{eq:lem3}). Hence, $\varphi_a(x_1, \cdots,x_d)$ is $C^{\infty}$ on $[0,1]\times[0,\frac{1}{2}]\times \cdots\times[0,\frac{1}{2}]$.
\end{proof}

In Lemma 4 and Lemma 5, we use
\[
  \partial^{\mbox{\boldmath $\alpha$}}=\frac{\partial^{\alpha_1}}{ \partial x_1^{ \alpha_1}} \cdots \frac{\partial^{\alpha_d}}{ \partial x_d^{ \alpha_d}}
\]
where $\mbox{\boldmath $\alpha$}=(\alpha_1, \alpha_2, \cdots, \alpha_d)\in \mathbb{Z}_{\geq 0}^d$.

\begin{lem}
Let $ \mbox{\boldmath $\alpha$}=(0, \alpha_2, \cdots, \alpha_d)\in \mathbb{Z}_{\geq 0}^d$. Then we have
\[
  \partial^{\mbox{\boldmath $\alpha$}}\left( \prod_{j=2}^{d}\frac{x_1t_2 \cdots t_j}{e^{x_1t_2 \cdots t_j}-1}\right)
  = \sum_{m}\prod_{j=2}^{d}\frac{f_{m,j}(e^{x_1t_2, \cdots,t_j},x_1,t_2, \cdots,t_j)}{\left(e^{x_1t_2 \cdots t_j}-1\right)^{ \alpha_2+ \cdots+ \alpha_j+1}}
\]
where $\sum_{m}$ is a finite summation, $f_{m,j}$ is a polynomial which degree of $e^{x_1t_2 \cdots t_j}$ is less than or equal to $ \alpha_2+ \cdots+ \alpha_j$.
\end{lem}

\begin{proof}
Induction on $|\mbox{\boldmath $\alpha$}|= \alpha_1+ \cdots+ \alpha_d$.$\\$
When $|\mbox{\boldmath $\alpha$}|=0$, Lemma 4 is trivial.

Suppose that Lemma 4 is valid for $(0, \alpha_2, \cdots, \alpha_d)$. Differentiating with respect to $x_k$, we have
\begin{align}
  \nonumber&\frac{\partial}{ \partial x_k} \partial^{ \mbox{\boldmath $\alpha$}}\left( \prod_{j=2}^{d}\frac{x_1t_2 \cdots t_j}{e^{x_1t_2 \cdots t_j}-1}\right)\\
  \nonumber=& \sum_{m}\frac{\partial}{ \partial x_k}\left(\prod_{j=2}^{d}\frac{f_{m,j}(e^{x_1t_2, \cdots,t_j},x_1,t_2, \cdots,t_j)}{\left(e^{x_1t_2 \cdots t_j}-1\right)^{ \alpha_2+ \cdots+ \alpha_j+1}}\right)\\
  =& \sum_{m}\left(\prod_{j=2}^{k-1}\frac{f_{m,j}(e^{x_1t_2, \cdots,t_j},x_1,t_2, \cdots,t_j)}{\left(e^{x_1t_2 \cdots t_j}-1\right)^{ \alpha_2+ \cdots+ \alpha_j+1}}\right)
  \frac{\partial}{ \partial x_k}\left(\prod_{j=k}^{d}\frac{f_{m,j}(e^{x_1t_2, \cdots,t_j},x_1,t_2, \cdots,t_j)}{\left(e^{x_1t_2 \cdots t_j}-1\right)^{ \alpha_2+ \cdots+ \alpha_j+1}}\right).
  \label{eq:thm4}
\end{align}
Using the product rule to (\ref{eq:thm4}), we find
\begin{multline}
  \frac{\partial}{ \partial x_k}\left(\prod_{j=k}^{d}\frac{f_{m,j}(e^{x_1t_2, \cdots,t_j},x_1,t_2, \cdots,t_j)}{\left(e^{x_1t_2 \cdots t_j}-1\right)^{ \alpha_2+ \cdots+ \alpha_j+1}}\right)\\
  = \sum_{l=k}^{d}\left(\prod_{\substack{j=k\\j\neq l}}^{d}\frac{f_{m,j}(e^{x_1t_2, \cdots,t_j},x_1,t_2, \cdots,t_j)}{\left(e^{x_1t_2 \cdots t_j}-1\right)^{ \alpha_2+ \cdots+ \alpha_j+1}}\right)
  \frac{\partial}{ \partial x_k}\left(\frac{f_{m,l}(e^{x_1t_2, \cdots,t_l},x_1,t_2, \cdots,t_l)}{\left(e^{x_1t_2 \cdots t_l}-1\right)^{ \alpha_2+ \cdots+ \alpha_l+1}}\right).
  \label{eq:thm42}
\end{multline}
Using the quotient rule to (\ref{eq:thm42}), we get
\begin{multline}
  \frac{\partial}{ \partial x_k}\left(\frac{f_{m,l}(e^{x_1t_2, \cdots,t_l},x_1,t_2, \cdots,t_l)}{\left(e^{x_1t_2 \cdots t_l}-1\right)^{ \alpha_2+ \cdots+ \alpha_l+1}}\right)\\
  =\pm\frac{(e^{x_1t_2 \cdots t_l}-1)\frac{\partial}{ \partial x_k}f_{m,l}-( \alpha_1+ \cdots+ \alpha_l+1)x_1t_2 \cdots t_{k-1}t_{k+1} \cdots t_l f_{m,l}}{(e^{x_1t_2 \cdots t_l}-1)^{ \alpha_2+\cdots+\alpha_l+2}},
  \label{eq:thm43}
\end{multline}
where the choice of $\pm$ depends on $a_k$. In the numerator of (\ref{eq:thm43}), the degree of $e^{x_1t_2 \cdots t_l}$ is less than or equal to $ \alpha_2+ \cdots+ \alpha_l+1$. Hence, Lemma 4 is valid for $(0, \alpha_2, \cdots,\alpha_k+1, \cdots, \alpha_d)$.
\end{proof}

\begin{lem}
For each $ \mbox{\boldmath $\alpha$}=(0, \alpha_2, \cdots, \alpha_d)\in \mathbb{Z}_{\geq 0}^d$, $\partial^{ \mbox{\boldmath $\alpha$}}\psi_a(x_1, \cdots, x_d)$ is bounded on $[0,\infty)\times[0,\frac{1}{2}]\times \cdots \times[0,\frac{1}{2}]$.
\end{lem}
\begin{proof}
By the Leibniz rule, we have
\[
  | \partial^{ \mbox{\boldmath $\alpha$}}\psi_a(x_1, \cdots,x_d)|
  \ \leq\ \sum_{\mbox{\boldmath $\beta$}\leq \mbox{\boldmath $\alpha$}}
  \begin{pmatrix}
  \mbox{\boldmath $\alpha$}\\ \mbox{\boldmath $\beta$}
  \end{pmatrix}
  \left| \partial^{ \mbox{\boldmath $\beta$}}\left(\prod_{j=2}^{d}(1-x_j)^{v_j}\right)\right|
  \left|\partial^{\mbox{\boldmath $\alpha$}-\mbox{\boldmath $\beta$}}\left( \prod_{j=2}^{d}\frac{x_1t_2 \cdots t_j}{e^{x_1t_2 \cdots t_j}-1}\right)\right|.
\]

By Lemma 4, $\left|\partial^{\mbox{\boldmath $\alpha$}-\mbox{\boldmath $\beta$}}\left(\prod_{j=2}^{d}x_1t_2 \cdots t_j/(e^{x_1t_2 \cdots t_j}-1)\right)\right|$ is bounded on $[0,\infty)\times[0,\frac{1}{2}]\times \cdots \times[0,\frac{1}{2}]$. Hence what we have to prove is that $\left| \partial^{ \mbox{\boldmath $\beta$}}\prod_{j=2}^{d}(1-x_j)^{v_j}\right|$ is bounded on $[0,\infty)\times[0,\frac{1}{2}]\times \cdots \times[0,\frac{1}{2}]$. We find
\begin{align}
  \nonumber\left| \partial^{ \mbox{\boldmath $\beta$}}\prod_{j=2}^{d}(1-x_j)^{v_j}\right|
  =& \prod_{j=2}^{d}\left| \partial^{ \beta_j}_{x_j}(1-x_j)^{v_j}\right|\\
  \nonumber=& \prod_{j=2}^{d}\left|(v_j-\beta_j+1)_{\beta_j}(1-x_j)^{v_j-\beta_j}\right|\\
  \nonumber\leq&\prod_{j=2}^{d}\left|(v_j-\beta_j+1)_{\beta_j}\right|\max\left\{1,\left|\frac{1}{2}\right|^{\Re(v_j-\beta_j)}\right\},
\end{align}
where $\mbox{\boldmath $\beta$}=(\beta_1, \cdots,\beta_d)$.
\end{proof}

\begin{lem}
Let $|a|,|b|\leq\frac{1}{2},\ a\neq0,b\neq0$. Then for each $m,n\in \mathbb{Z}$, we have
\[
\frac{(a)_n}{(b)_m}=\left\{ \begin{array}{ll}
\frac{a}{b}\left(\frac{(n-1)!}{(m-1)!}+O(a)+O(b)\right)&(n\geq m\geq 1),\\
O(a)  &(n\geq1>m),\\
(-1)^{m+n}\frac{(-m)!}{(-n)!}+O(a)+O(b)&(1>n\geq m)\\
\end{array} \right.
\]
as $\ a,b\to0$.
\end{lem}

\begin{proof}
In the case $n\geq m\geq 1$, we have
\begin{align}
  \nonumber\frac{(a)_n}{(b)_m}=&\frac{a}{b}\left(\frac{(a+1) \cdots(a+n-1)}{(b+1) \cdots(b+m-1)}\right)\\
  \nonumber=&\frac{a}{b}\left\{\bigg((n-1)!+O(a)\bigg)\left(\frac{1}{(m-1)!}+O(b)\right)\right\}\\
  \nonumber=&\frac{a}{b}\bigg(\frac{(n-1)!}{(m-1)!}+O(a)+O(b)\bigg).
\end{align}
In the case $n\geq1>m$, we have
\begin{align}
  \nonumber\frac{(a)_n}{(b)_m}=&a(a+1)\cdots(a+n-1)(b-1) \cdots(b+m)\\
  \nonumber\ll& a.
\end{align}
In the case $1>n\geq m$, we have
\begin{align}
  \nonumber\frac{(a)_n}{(b)_m}&=\frac{(b-1) \cdots(b+m)}{(a-1) \cdots(a+n)}\\
  \nonumber&=\big\{(-1)^m(-m)!+O(b)\big\}\left\{(-1)^n(-n)!^{-1}+O(a)\right\}\\
  \nonumber&=(-1)^{m+n}\frac{(-m)!}{(-n)!}+O(a)+O(b).
\end{align}

\end{proof}

\section{Proof of Theorem 1}

In this section, we prove Theorem 1.

By [\cite{zhao}, p1279, (7)], we have
\begin{multline}
  \Gamma(s_1) \cdots \Gamma(s_d)\zeta_d(s_1, \cdots,s_d)\\
  =\int_0^1 \cdots\int_0^1\int_0^{\infty}\prod_{j=1}^{d}x_j^{s_d(j)-d+j-2} \prod_{j=2}^{d}(1-x_j)^{s_{j-1}-1} \prod_{j=1}^{d}\frac{x_1 \cdots x_j}{e^{x_1 \cdots x_j}-1}dx_1 \cdots dx_d.
  \label{eq:zhao}
\end{multline}
The right-hand side of (\ref{eq:zhao}) is divided into two terms,
\begin{equation}
  \int_0^1 \cdots \int_0^1 \int_0^{\infty}=\int_0^1 \cdots \int_0^1 \int_0^{1}+\int_0^1 \cdots \int_0^1 \int_1^{\infty}.
  \label{eq:ab}
\end{equation}
First we consider the first term of (\ref{eq:ab}). By $x/(e^x-1)=\sum_{m=0}^{\infty}(B_m/m!)x^m\ (|x|<2\pi)$, we see that the first term is
\begin{equation}
  \int_0^1 \cdots \int_0^1\prod_{j=1}^{d}x_j^{s_d(j)-d+j-2} \prod_{j=2}^{d}(1-x_j)^{s_{j-1}-1} \prod_{j=1}^{d}\left(\sum_{k=0}^{\infty}\frac{B_k}{k!}(x_1 \cdots x_j)^k\right)dx_1 \cdots dx_d.
\label{eq:a}
\end{equation}
Further we divide the summation in (\ref{eq:a}) as
\begin{align}
  \nonumber\prod_{j=1}^{d}\sum_{k=0}^{\infty}\frac{B_k}{k!}(x_1 \cdots x_j)^k=&
  \sum_{k=0}^{\infty}\sum_{p_1+ \cdots +p_d=k}\frac{B_{p_1} \cdots B_{p_d}}{p_1! \cdots p_d!}x_1^kx_2^{p_d(2)} \cdots x_d^{p_d(d)}\\
  \nonumber=&\sum_{k=0}^{n_1}\sum_{p_1+ \cdots +p_d=k}\frac{B_{p_1} \cdots B_{p_d}}{p_1! \cdots p_d!}x_1^kx_2^{p_d(2)} \cdots x_d^{p_d(d)}\\
  &+\sum_{k=n_1+1}^{\infty}\sum_{p_1+ \cdots +p_d=k}\frac{B_{p_1} \cdots B_{p_d}}{p_1! \cdots p_d!}x_1^kx_2^{p_d(2)} \cdots x_d^{p_d(d)}.
  \label{eq:a12}
\end{align}
The contribution of the first term of (\ref{eq:a12}) is
\begin{equation}
  \sum_{k=0}^{n_1}\sum_{p_1+ \cdots +p_d=k}\frac{B_{p_1} \cdots B_{p_d}}{p_1! \cdots p_d!}
  \frac{1}{s_d(1)-d+k} \prod_{j=2}^{d}B(s_d(j)-d+j+p_d(j)-1,s_{j-1}).
  \label{eq:a1}
\end{equation}
This is the first term of (\ref{eq:thm1}). Changing the order of integration of the second term of (\ref{eq:a12}), we have
\begin{multline}
  \int_0^1x_1^{s_d(1)-d+n_1}\int_0^1 \cdots\int_0^1 
  \left( \prod_{j=2}^{d}x_j^{s_d(j)-d+j-2}(1-x_j)^{s_{j-1}-1}\right)\\
  \left(\sum_{k=n_1+1}^{\infty}\sum_{p_1+ \cdots+p_d=k}
  \frac{B_{p_1} \cdots B_{p_d}}{p_1! \cdots p_d!}x_1^{k-n_1-1}x_2^{p_d(2)} \cdots x_d^{p_d(d)}\right)
  dx_2 \cdots dx_ddx_1.\label{eq:cc}
\end{multline}
Dividing the integral into $\int_0^{1/2}$ and $\int_{1/2}^1$ for $x_2, \cdots,x_d$, we have
\[
  \int_0^1 \cdots \int_0^1=
  \sum_{(a_2, \cdots,a_d)}\int_{\frac{a_2}{2}}^{\frac{a_2+1}{2}}\cdots\int_{\frac{a_d}{2}}^{\frac{a_d+1}{2}},
\]
where the notation $ \sum_{(a_2, \cdots,a_d)}$ is defined in the statement of Theorem 1. Changing variables, we find that (\ref{eq:cc}) is
\begin{multline}
  \nonumber\int_0^1x_1^{s_d(1)-d+n_1}\Bigg\{\sum_{(a_2,\cdots,a_d)}\int_0^{\frac{1}{2}}\cdots\int_0^{\frac{1}{2}} 
  \left( \prod_{j=2}^{d}x_j^{u_j}(1-x_j)^{v_j}\right)\\
  \nonumber\left(\sum_{k=n_1+1}^{\infty}\sum_{p_1+ \cdots+p_d=k}
  \frac{B_{p_1} \cdots B_{p_d}}{p_1! \cdots p_d!}x_1^{k-n_1-1}t_2^{p_d(2)} \cdots t_d^{p_d(d)}\right)
  dx_2 \cdots dx_d\Bigg\}dx_1.
\end{multline}
By Lemma 2, Lemma 3 and the definition of $ \varphi_a(x_1, \cdots,x_d)$, we find that the above is
\begin{equation}
  \int_0^1x_1^{s_d(1)-d+n_1}F_{ \varphi_a}(x_1)dx_1.
  \label{eq:a2}
\end{equation}
By (\ref{eq:a1}) and (\ref{eq:a2}), we see that the first term of (\ref{eq:ab}) is
\begin{align}
  \nonumber&\sum_{k=0}^{n_1}\sum_{p_1+ \cdots +p_d=k}\frac{B_{p_1} \cdots B_{p_d}}{p_1! \cdots p_d!}
  \frac{1}{s_d(1)-d+k} \prod_{j=2}^{d}B(s_d(j)-d+j+p_d(j)-1,s_{j-1})\\
  +&\int_0^1x_1^{s_d(1)-d+n_1}F_{ \varphi_a}(x_1)dx_1.
  \label{eq:A}
\end{align}

Next, we consider the second term of (\ref{eq:ab}). Similarly to the deformation of (\ref{eq:cc}), we have
\[
  \int_1^{\infty}\frac{x_1^{s_d(1)-d}}{e^{x_1}-1}\left(\sum_{(a_2,\cdots,a_d)}\int_0^{\frac{1}{2}}\cdots\int_0^{\frac{1}{2}} 
  \prod_{j=2}^{d}x_j^{u_j}(1-x_j)^{v_j}
  \frac{x_1t_2 \cdots t_j}{e^{x_1t_2 \cdots t_j}-1}
  dx_2 \cdots dx_d\right)dx_1.
\]
Using Lemma 2, we find that the above is

\begin{equation}
  \int_1^{\infty}\frac{x_1^{s_d(1)-d}}{e^{x_1}-1}F_{\psi_a}(x_1)dx_1.
  \label{eq:B}
\end{equation}
By (\ref{eq:A}) and (\ref{eq:B}), we obtain (\ref{eq:thm1}). 

Now let us consider when (\ref{eq:thm1}) is holomorphic.
The first term is holomorphic when
\begin{align}
  \nonumber &s_d(1)\neq d,d-1, \cdots,d-n_1, &&\\
  \nonumber &s_d(j)\neq d-j+1,d-j,d-j-1, \cdots &&(j=2, \cdots,d),\\
  \nonumber &s_j\neq0,-1,-2, \cdots &&(j=1, \cdots,d-1).
\end{align}
By Lemma 3, the second term is holomorphic when
\begin{align}
  \nonumber &s_d(j)\neq d-j+1,d-j, \cdots,d-j+1-n_j &&(j=2, \cdots,d),\\
  \nonumber &s_j\neq0,-1, \cdots,-n_j &&(j=2, \cdots,d-1),\\
  \nonumber &\Re(s_d(j))>d-j-n_j&&(j=1, \cdots,d),\\
  \nonumber &\Re(s_{j-1})>-n_j-1&&(j=2,\cdots,d).
\end{align}
By Lemma 5, the third term is holomorphic when
\begin{align}
  \nonumber &s_d(j)\neq d-j+1,d-j, \cdots,d-j+1-n_j &&(j=2, \cdots,d),\\
  \nonumber &s_j\neq0,-1, \cdots,-n_j &&(j=2, \cdots,d-1),\\
  \nonumber &\Re(s_d(j))>d-j-n_j&&(j=2, \cdots,d),\\
  \nonumber &\Re(s_{j-1})>-n_j-1&&(j=2,\cdots,d).
\end{align}
Hence, we obtain Theorem 1.

\section{Proof of Theorem 2}

In this section, we prove Theorem 2. If $d=1$, $\zeta_1(s_1)$ is Riemann zeta function. Hence, Theorem 2 is clear. So we prove Theorem 2 in the case $d>1$. Suppose that $m_j,\ \varepsilon_j\ (j=1, \cdots,d)$ and $M$ are defined in the statement of Theorem 2. We use (\ref{eq:thm1}) with $s_j=-m_j+ \varepsilon_j\ (j=1 \cdots,d)$ and $n_1= \cdots=n_d=M+d$.

First, we estimate the second term and the third term. When $ (\varepsilon_1, \cdots, \varepsilon_d)\to(0, \cdots,0)$, these terms are bounded except $(u_{\sigma(j)}+1)_{n_{\sigma(j)}+1}^{-1}$ and $(u_{\sigma(j)}+1)_{k_{\sigma(j)}+1}^{-1}$. Hence, we have
\[
  \int_0^1 x_1^{s_d(1)-d+n_1}F_{\varphi_a}(x_1)dx_1
  +\int_1^{\infty} \frac{x_1^{s_d(1)-d}}{e^{x_1}-1}F_{\psi_a}(x_1)dx_1
  = \sum_{(a_2, \cdots,a_d)}O\left( \prod_{j=2}^{d}w_j^{-1}\right)
\]
where
\[
  w_j:=\begin{cases}
  \varepsilon_d(j)& (a_j=0)\\
  \varepsilon_{j-1}& (a_j=1).
\end{cases}
\]
On the other hand, using $ \Gamma(z) \Gamma(1-z)=\pi/\sin(\pi z)$, we can estimate
\begin{align}
  \nonumber\frac{1}{\Gamma(s_1)\cdots\Gamma(s_d)}\ll&\sin(\pi s_1) \cdots \sin(\pi s_d)\\
  \nonumber\ll&\sin(\pi \varepsilon_1) \cdots\sin(\pi \varepsilon_d)\\
  \nonumber\ll&\varepsilon_1 \cdots \varepsilon_d.
\end{align}
Then, we have
\begin{align}
  \nonumber&\frac{1}{\Gamma(s_1)\cdots\Gamma(s_d)}\left(
  \int_0^1 x_1^{s_d(1)-d+n_1}F_{\varphi_a}(x_1)dx_1
  +\int_1^{\infty} \frac{x_1^{s_d(1)-d}}{e^{x_1}-1}F_{\psi_a}(x_1)dx_1\right)\\
  \nonumber=&\sum_{(a_1, \cdots,a_d)}O\left(\left( \prod_{j=2}^{d}\frac{ \varepsilon_{j-1}}{w_j}\right) \varepsilon_d \right)\\
  \nonumber=&\sum_{(a_1, \cdots,a_d)}O\left(\left( \prod_{\substack{j=2\\a_j=0}}^{d}\frac{ \varepsilon_{j-1}}{\varepsilon_d(j)}\right) \varepsilon_d \right).
\end{align}
Since $\varepsilon_k/ \varepsilon_d(j)\ll1(j=1, \cdots,d,\ k=j, \cdots,d)$, we obtain
\begin{equation}
  \frac{1}{\Gamma(s_1)\cdots\Gamma(s_d)}\left(
  \int_0^1 x_1^{s_d(1)-d+n_1}F_{\varphi_a}(x_1)dx_1
  +\int_1^{\infty} \frac{x_1^{s_d(1)-d}}{e^{x_1}-1}F_{\psi_a}(x_1)dx_1\right)
  = \sum_{j=1}^{d}O( \varepsilon_j).
 \label{eq:23}
\end{equation}

Next, we estimate the first term of (\ref{eq:thm1}). First, we estimate the factors containing gamma funcions and beta functions as
\begin{align}
  \nonumber&\frac{1}{\Gamma(s_1) \cdots \Gamma(s_d)}
  \prod_{j=2}^{d}B(s_d(j)-d+j+p_d(j)-1,s_{j-1})\\
  \nonumber=&\frac{1}{ \Gamma(s_d)} \prod_{j=2}^{d}
  \frac{ \Gamma(s_d(j)-d+j+p_d(j)-1)}{ \Gamma(s_d(j-1)-d+j+p_d(j)-1)}\\
  =&\frac{1}{( \varepsilon_d)_{-m_d} \Gamma( \varepsilon_d(1))}
  \prod_{j=2}^{d}\frac{(\varepsilon_d(j))_{s_d(j)-d+j+p_d(j)-1}}{(\varepsilon_d(j-1))_{s_d(j-1)-d+j+p_d(j)-1}}.
  \label{eq:gabe}
\end{align}
By Lemma 6, we have
\begin{align}
  \nonumber\frac{1}{( \varepsilon_d)_{-m_d} \Gamma( \varepsilon_d(1))}
  =&\bigl((-1)^{m_d}m_d!+O( \varepsilon_d)\bigr)
  \left(\frac{\sin(\pi \varepsilon_d(1))}{\pi}\Gamma(1- \varepsilon_d(1))\right)\\
  \nonumber=&(-1)^{m_d}m_d! \varepsilon_d(1)+O( \varepsilon_d(1)^2)+O( \varepsilon_d(1) \varepsilon_d)
\end{align}
and
\begin{multline}
  \prod_{j=2}^{d}\frac{(\varepsilon_d(j))_{s_d(j)-d+j+p_d(j)-1}}{(\varepsilon_d(j-1))_{s_d(j-1)-d+j+p_d(j)-1}}\\
  = \prod_{j=2}^{d}\left(h(-m_d(j)-d+j+p_d(j)-1,-m_d(j-1)-d+j+p_d(j)-1)
   \frac{[ \varepsilon_d(j)]_{-m_d(j)-d+j+p_d(j)-1}}{[ \varepsilon_d(j-1)]_{-m_d(j-1)-d+j+p_d(j)-1}}\right)\\
  \nonumber+\sum_{j=2}^{d}\left\{
  O\left(\frac{ \varepsilon_d(j)}{ \varepsilon_d(j-1)} \varepsilon_d(j)\right)
  +O\bigl(\varepsilon_d(j-1)\bigr)+O\bigl( \varepsilon_d(j)\bigr)\right\},
\end{multline}
hence, we find (\ref{eq:gabe}) is
\begin{multline}
  \frac{1}{\Gamma(s_1) \cdots \Gamma(s_d)}
  \prod_{j=2}^{d}B(s_d(j)-d+j+p_d(j)-1,s_{j-1})=(-1)^{m_d}m_d!\varepsilon_d(1)\times \\
  \times\prod_{j=2}^{d}\left(h(-m_d(j)-d+j+p_d(j)-1,-m_d(j-1)-d+j+p_d(j)-1)
   \frac{[ \varepsilon_d(j)]_{-m_d(j)-d+j+p_d(j)-1}}{[ \varepsilon_d(j-1)]_{-m_d(j-1)-d+j+p_d(j)-1}}\right)\\
  \label{eq:gabe2}+ \sum_{j=1}^{d}O\bigl( \varepsilon_j \varepsilon_d(1)\bigr).
\end{multline}
Using (\ref{eq:gabe2}), we can estimate the first term of (\ref{eq:thm1}),
\begin{multline}
  \nonumber\sum_{k=0}^{n_1}\sum_{p_1+ \cdots+p_d=k}\frac{B_{p_1}\cdots B_{p_d}}{p_1! \cdots p_d!}\frac{1}{s_d(1)-d+k}\Bigg\{(-1)^{m_d}m_d!\varepsilon_d(1)\times \\
  \times\prod_{j=2}^{d}\left(h(-m_d(j)-d+j+p_d(j)-1,-m_d(j-1)-d+j+p_d(j)-1)
  \nonumber \frac{[ \varepsilon_d(j)]_{-m_d(j)-d+j+p_d(j)-1}}{[ \varepsilon_d(j-1)]_{-m_d(j-1)-d+j+p_d(j)-1}}\right)\\
  \nonumber+ \sum_{j=1}^{d}O\bigl( \varepsilon_d(j) \varepsilon_d(1)\bigr)\Bigg\}.
\end{multline}
Using
\begin{equation}
  \nonumber \frac{1}{s_d(1)-d+k}=\begin{cases}
  O(1)&(k<n_1=M+d)\\
  \varepsilon_d(1)^{-1}&(k=n_1=M+d),
\end{cases}
\end{equation}
we have
\begin{multline}
(-1)^{m_d}m_d! \sum_{p_1+ \cdots+p_d=d+M}\frac{B_{p_1} \cdots B_{p_d}}{p_1! \cdots p_d!}\prod_{j=2}^{d}h(-m_d(j)-d+j+p_d(j)-1,-m_d(j-1)-d+j+p_d(j)-1)\\
  \times \frac{[ \varepsilon_d(j)]_{-m_d(j)-d+j+p_d(j)-1}}{[ \varepsilon_d(j-1)]_{-m_d(j-1)-d+j+p_d(j)-1}}+\sum_{j=1}^{d}O( \varepsilon_j).
  \label{eq:1}
\end{multline}

From (\ref{eq:23}) and (\ref{eq:1}), we obtain Theorem 2.

Graduate School of Mathematics\\
Nagoya University\\
Chikusa-ku, Nagoya 464-8602, Japan\\
E-mail:\ m11022v@math.nagoya-u.ac.jp

\end{document}